\documentclass[flen,a4paper]{article}
\usepackage{amsmath,amsthm,amssymb,amscd}
\usepackage{graphicx}
\usepackage{graphics}
\usepackage{verbatim}
\usepackage{enumerate}
\usepackage{mathrsfs}
\usepackage{color}
\usepackage[top=1in, bottom=0.8in, left=0.8in, right=0.8in]{geometry}

\makeatletter
\newcommand{\rmnum}[1]{\romannumeral #1}
\newcommand{\Rmnum}[1]{\expandafter\@slowromancap\romannumeral #1@}

\usepackage{color}

\newtheorem{theorem}{Theorem}[section]

\newtheorem{lemma}[theorem]{Lemma}
\newtheorem{definition}[theorem]{Definition}
\newtheorem{corollary}[theorem]{Corollary}
\newtheorem{claim}[theorem]{Claim}

\newtheorem{remark}[theorem]{Remark}
\makeatother

\newcommand{\intt}{{\rm int}}
\newcommand{\ext}{{\rm ext}}

\begin{document}
	
	\title{DP-3-colorability of planar graphs without cycles of length 4, 7 or 9}

	\vspace{3cm}
	\author{Yingli Kang\footnotemark[1]~, Ligang Jin\footnotemark[2]~, Xuding Zhu\footnotemark[2]}
	\footnotetext[1]{Department of Mathematics, Jinhua Polytechnic, Western Haitang  Road 888, 321017 Jinhua, China; Email: ylk8mandy@126.com}	
	\footnotetext[2]{Department of Mathematics,
		Zhejiang Normal University, Yingbin Road 688,
		321004 Jinhua,
		China; 
		Email: ligang.jin@zjnu.cn (Ligang Jin), xdzhu@zjnu.edu.cn (Xuding Zhu)}
	\date{}
	
	\maketitle

\begin{abstract}
This paper proves that every planar graph without cycles of length 4, 7, or 9 is DP-3-colorable.
\end{abstract}


\section{Introduction}
Graphs considered in this paper are finite and simple. 
A graph is \emph{planar} if it is embeddable into the   plane and a planar graph   embedded in the plane is called a \emph{plane graph}.

The problem of  3-coloring  planar graphs  with restriction on the lengths of cycles is a central topic in chromatic graph   theory. The classical  Gr\"{o}tzsch Theorem says that planar graphs without triangles are 3-colorable \cite{Grotzsch1959109}. It was conjectured by Steinberg    that planar graphs without cycles of lengths 4 and 5 are 3-colorable. This conjecture was disproved in      \cite{CA-disprove}. However, it motivated a lot of research in this area.  It is known that planar graphs with no cycles of lengths from 4 to 7 are 3-colorable \cite{BorodinEtc2005303}, and   it remains a challenging  open problem   whether   planar graphs without cycles of lengths from 4 to 6 are 3-colorable.  It is also known that if $4\textless i\textless j\textless 10$ and $\{ i,j\}\notin \{\{ 5,6\}, \{ 7,8\}$, $\{ 5,9\}, \{ 8,9\}\}$ \cite{BorodinEtc2005303,BorodinEtc2009668,Jinetc_2017_458,Jin_2016_469,LWWBMR-2009,WangChen20071552,Xu2009347}, then planar graphs without cycles of length from $\{4,i,j\}$ are 3-colorable.

The concept  of DP-coloring (also known as correspondence coloring)  of a graph was introduced by Dvo\v{r}\'{a}k and Postle \cite{Dvorak-Postle-2018}, who used it as a tool to prove that every planar graph without cycles of length from 4 to 8 is   3-choosable. 
Currently, as an independent coloring parameter, DP-coloring of graphs has attracted a lot of attention. 


  Dvo\v{r}\'{a}k and Postle \cite{Dvorak-Postle-2018} noted that Thomassen's proof \cite{Thomassen-1995} for choosability can be used to show that every planar graph without cycles of length from $\{3,4\}$ is DP-3-colorable.
Denote by $d^{\triangle}$ the smallest distance between triangles.
The DP-3-colorability was confirmed for planar graphs with $d^{\triangle}\geq 3$ and without cycles of length from $\{4,5\}$ \cite{Yin-Yu-2019}, planar graphs with $d^{\triangle}\geq 2$ and without cycles of length from $\{4,5,6\}$ \cite{Yin-Yu-2019}, $\{4,5,7\}$ \cite{Rao-Wang}, $\{5,6,7\}$ \cite{Liu-Loeb-Yin-Yu-2019}, or $\{5,6,8\}$ \cite{Rao-Wang}, planar graphs with neither intersecting triangles nor cycles of length from $\{4,5,6,7\}$ \cite{Lv-2022}, planar graphs with neither adjacent triangles nor cycles of length from $\{5,6,9\}$ \cite{Rao-Wang}, and planar graphs without cycles of length from $\{3,5,6\}$ \cite{Liu-Loeb-Yin-Yu-2019}, $\{3,6,7,8\}$ \cite{Liu-Loeb-Yin-Yu-2019}, $\{4,5,6,9\}$ \cite{Liu-Loeb-Yin-Yu-2019}, $\{4,5,7,9\}$ \cite{Liu-Loeb-Yin-Yu-2019}, $\{4,6,7,9\}$ \cite{Liu-Loeb-Rolek-Yin-Yu-2019}, $\{4,6,8,9\}$ \cite{Liu-Loeb-Rolek-Yin-Yu-2019}, or $\{4,7,8,9\}$ \cite{Liu-Loeb-Rolek-Yin-Yu-2019}.    

In this paper, we focus on classes of planar graphs forbidden cycle lengths of type $\{4,i,j\}$ with $5 \le i < j$. Before this paper, no such family of graphs is known to be DP-3-colorable. This paper proves the following result:
 
\begin{theorem} \label{thm479}
	Every planar graph without cycles of length 4, 7 or 9 is DP-3-colorable.
\end{theorem}

It is well-known that DP-$k$-colorable graphs are signed  $k$-choosable \cite{JWZ_2021,KS_2015,Mac_Ras_Sko_2015}. Hence, we have the following corollary:

\begin{corollary} \label{coro479-list}
	Every planar graph without cycles of length 4, 7 or 9 is 3-choosable  (and hence $3$-colorable).
\end{corollary}

\section{$S$-colorings, bad cycles, and other notations}

\begin{definition}
	Assume $G$ is a graph and $S$ is a set of permutations of positive integers. An \emph{$S$-labelling} of $G$ is a pair $(D,\sigma)$, where $D$ is an orientation of $G$ and $\sigma \colon\ E(D)\rightarrow S$ is a mapping which assigns to each arc $e$ of $D$ a permutation $\sigma_e$ of $S$. The mapping $\sigma$ is called the \emph{signature} of $G$, and $\sigma_e$ the \emph{sign} of $e$. The pair $(D,\sigma)$ is also called an \emph{$S$-labelled graph}. A \emph{$k$-coloring} of $(D,\sigma)$ is a mapping $f\colon\ V(G)\rightarrow [k]=\{1,2,\cdots,k\}$ such that for each arc $e=(x,y)$ of $D$, $\sigma_e(f(x))\neq f(y)$. A graph $G$ is \emph{$S$-$k$-colorable} if $(D,\sigma)$ is $k$-colorable for every $S$-labelling $(D,\sigma)$ of $G$. The \emph{$S$-chromatic number} of a graph $G$ is the minimum integer $k$ such that $G$ is $S$-$k$-colorable.
\end{definition}

The concept of $S$-$k$-coloring of a graph was introduced by Jin, Wong, and Zhu \cite{JWZ_2021}. It is a common generalization of many coloring concepts including $k$-coloring, signed-$k$-coloring, signed-$Z_k$-coloring, list-$k$-coloring, DP-$k$-coloring, group coloring, and coloring of gain graphs. In particular, if $S$ is   the set of all permutations of integers, then $S$-$k$-colorable is equivalent to 
DP-$k$-colorable. In this paper, we use this alternate definition of DP-$k$-colorable graphs. 

Let $S_k$ denote the set of all the permutations on the set $[k]=\{1,2,,\ldots,k\}$. 
For an  $S$-$k$-coloring of a graph, since the color set is $[k]$, the set $S$ can be restricted to 
permutations on $[k]$. So, a DP-$k$-coloring is equivalent to an $S_k$-$k$-coloring.
Therefore, Theorem \ref{thm479} can be reformulated as follows.
\begin{theorem} \label{thm479-S}
	Every planar graph without cycles of length 4, 7 or 9 is $S_3$-3-colorable.
\end{theorem}

Given an $S_k$-labelled graph $(D,\sigma)$, if $D'$ is obtained from $D$ by reversing an arc $(x,y)$, and $\sigma'$ is obtained from $\sigma$ by 
letting $\sigma'(y,x)= (\sigma(x,y))^{-1}$, then $(D, \sigma)$ is equivalent to $(D',\sigma')$, and $(D,\sigma)$ is $S$-$k$-colorable if and only if $(D',\sigma')$ is $S$-$k$-colorable. Given $(D, \sigma)$, 
\emph{switch} a vertex $u$ by a sign $s\in S_k$ defines another signature $\sigma'$ as follows: 
\begin{align*}
	\sigma'_e=
	\begin{cases}
		s \circ \sigma_e, &\text{~if $e=(u,v)$;}\\
		\sigma_e \circ s^{-1}, &\text{~if $e=(v,u)$;}\\
		\sigma_e, &\text{~otherwise.}
	\end{cases}
\end{align*}
Two $S_k$-labelled graphs are \emph{switch-equivalent} if one can be obtained from the other by a sequence of switches.
Clearly, two switch-equivalent $S_k$-labelled graphs have the same $S_k$-chromatic number.
An edge is $\textit{positive}$ if its sign is $id$ (the identity); \textit{negative} otherwise.
A cycle is \textit{all-positive} if all of its edges are positive. A cycle is \textit{positive} if it is switch-equivalent to an all-positive cycle; \textit{negative} otherwise.

Let $\mathcal{G}$ denote the set of  connected plane graphs without cycles of length 4, 7, or 9. 
For a set of vertices or a set of edges $S$ of a graph $G$, denote by $G[S]$ the subgraph of $G$ induced by $S$.
\begin{definition}
	Let $G$ be a graph in $\mathcal{G}$ and $C$ be a cycle in $G$.
	If a vertex $v\notin V(C)$ has three neighbors $v_1,v_2,v_3$ on $C$, then $G[\{vv_1,vv_2, vv_3\}]$ is called a \emph{claw} of $C$.
	If $u_1,u_2\notin V(C)$ are two adjacent vertices such that for each $i\in\{1,2\}$, $u_i$ has two neighbors $u_i'$ and $u_i''$ on $C$ (it allows that $\{u_1',u_1''\}\cap \{u_2',u_2''\}\neq \emptyset$), then $G[\{u_1u_2,u_1u_1',u_1u_1'',u_2u_2',u_2u_2''\}]$ is called a \emph{biclaw} of $C$ (see Figure \ref{fig_claw}). A \emph{good cycle} is a $13^-$-cycle that has neither claws nor biclaws.
	A \emph{bad cycle} is a $13^-$-cycle that is not good.
	\begin{figure}[h]
		\centering
		\includegraphics[width=5cm]{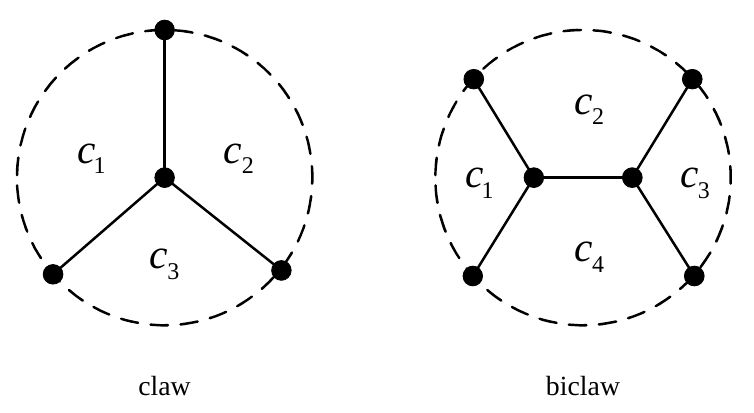}\\
		\caption{A claw or biclaw of a cycle}\label{fig_claw}
	\end{figure}
\end{definition}

As in Figure \ref{fig_claw}, the cycles into which a claw or a biclaw divides $C$ are called \emph{cells}.
Let $c_i$ be the length of a cell.
We further call $C$ a \emph{$(c_1,c_2,c_3)$-claw} or a \emph{$(c_1,c_2,c_3,c_4)$-biclaw}.

By the definition of bad cycles, one can easily conclude the following lemma.
\begin{lemma} \label{bad cycle}
	If $C$ is a bad cycle of a graph of $\cal{G}$, then $|C|\in\{12,13\}$. Furthermore, if $|C|=12$, then $C$ has a (3,5,10)-claw, (5,5,8)-claw, or (6,6,6)-claw; if $|C|=13$, then $C$ has a (3,5,11)-claw or (5,5,5,8)-biclaw, see Figure \ref{fig_bad_cycle}. 
	\begin{figure} [hh]
		\centering
		\includegraphics[width=3in]{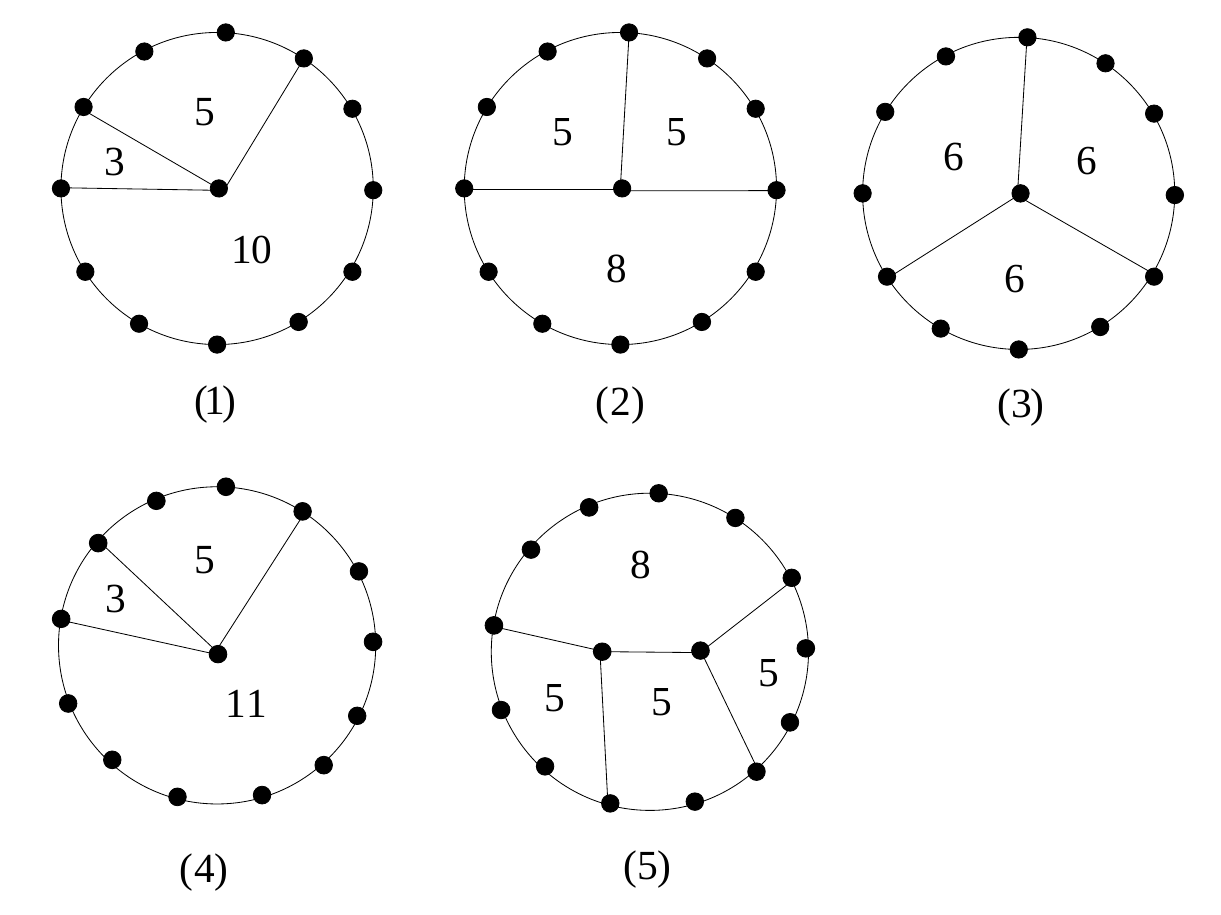}
		\caption{Bad cycles in a graph of $\cal{G}$}\label{fig_bad_cycle}
	\end{figure}
\end{lemma}

The \textit{length} of a path $P$ or a cycle $C$, denoted by $|P|$ or $|C|$, is the number of edges of $P$ or $C$, respectively. Denote by $d(f)$ the size of a face $f$ of a plane graph.  
Let $k$ be a positive integer. A \textit{$k$-vertex} (resp., $k^+$-vertex and $k^-$-vertex) is a vertex $v$ with $d(v)=k$ (resp., $d(v)\geq k$ and $d(v)\leq k$). 
Similar notations are applied for paths, cycles, and faces, where $d(v)$ is replaced by $|P|,|C|$, and $d(f)$, respectively.
A path $P$ and a vertex $v\notin V(P)$ are \emph{adjacent} if $v$ is adjacent to an end-vertex of $P$.
A \emph{$k$-string} is a path on $k$ 2-vertices that is adjacent to no 2-vertices.
A $k$-cycle on vertices $v_1,\ldots,v_k$ in clockwise cyclic order is denoted by $[v_1\ldots v_k]$.
If $u$ is a 3-vertex of a triangle $T$, then the neighbor of $u$ not on $T$ is called the \emph{outer neighbor} of $u$ (also, of $T$). 
A \emph{splitting path} of a cycle $C$ is a path $P=u_1u_2\ldots u_k$ with $V(C)\cap V(P)=\{u_1,u_k\}$.

Consider a plane graph $G$.
A vertex is \textit{external} if it lies on the boundary of the unbounded face; \textit{internal} otherwise.  
For a cycle $C$, let $\intt(C)$ and $\ext(C)$ denote the set of vertices in the interior and exterior of  $C$, respectively. 
A cycle $C$ is \textit{separating} if both $\intt(C)$ and $\ext(C)$ are nonempty.
Denote by $\intt[C]$ (resp., $\ext[C]$) the subgraph of $G$ consisting of $C$ and its interior (resp., $C$ and its exterior).

Given $G\in \mathcal{G}$,
if $H$ is the union of a 5-face $[v_2v_3\ldots v_6]$ and a 3-face $[v_2v_1v_6]$ which are adjacent, then $H$ is called a \textit{special subgraph} of $G$, and both faces of $H$ are called $\textit{special faces}$.
Since $G\in \mathcal{G}$, any two special subgraphs of $G$ are edge-disjoint, but they may share vertices and if so, it is easy to verify that they share precisely one vertex.
Denote by $\mathcal{H}(G)$ the union of all the special subgraphs of $G$.
A face is \textit{light} if its incident vertices are all internal 3-vertices.
A \textit{bad 3-face} is a positive light non-special 3-face, and the vertices of a bad 3-face are called \textit{bad vertices}.

\section{The proof of Theorem \ref{thm479}} \label{sec_proof}
Theorem \ref{thm479} (equivalently, Theorem \ref{thm479-S}) follows from the following theorem.
\begin{theorem} \label{thm_main_extension}
	Let $(D,\sigma)$ be an $S_3$-labelling of a graph $G\in \cal{G}$. If the boundary of the unbounded face $f_0$ of $G$ is a good cycle, then every 3-coloring of $(D[V(f_0)],\sigma)$  extends to a 3-coloring of $(D,\sigma)$.
\end{theorem}
To see that Theorem \ref{thm479-S} follows from Theorem \ref{thm_main_extension},
take any $S_3$-labelling $(D,\sigma)$ of a graph $G\in \mathcal{G}$.
If $D$ has no triangles, then it has girth at least 5 and is known to be DP-$3$-colorable \cite{Dvorak-Postle-2018};
otherwise, take a triangle $T$ of $D$ and a 3-coloring $\phi$ of $(T,\sigma)$. Lemma \ref{bad cycle} implies that $T$ is good. So, $\phi$ extends to both $(\ext[T],\sigma)$ and $(\intt[T],\sigma)$ by Theorem \ref{thm_main_extension}.
This results in a 3-coloring of $(D,\sigma)$.

We shall prove Theorem \ref{thm_main_extension} by contradiction.
Suppose to the contrary that Theorem \ref{thm_main_extension} is false.
Let $(D,\sigma)$ be a counterexample to Theorem \ref{thm_main_extension} with minimum $|V(G)|+|E(G)|$, where $G$ is the underlying graph of $D$.
So, we may assume that the boundary $U$ of $f_0$ is a good cycle, and there exists a 3-coloring $\phi_0$ of $(D[V(f_0)],\sigma)$ which is not extendable to $(D,\sigma)$.

By the minimality of $(D,\sigma)$, cycle $U$ has no chords.

\subsection{Structural properties}
\begin{lemma} \label{lem_separating-cycle}
	$G$ has no separating good cycles.
\end{lemma}

\begin{proof}
	If $C$ is a  separating good cycle of $G$, then by the minimality of $(D,\sigma)$, we can extend $\phi_0$ to $(D-\intt(C),\sigma)$ and the resulting coloring of $C$ to $(\intt[C],\sigma)$. This results in an extension of $\phi_0$ to $(D,\sigma)$, a contradiction.
\end{proof}

\begin{lemma}\label{lem_2connected}
	$G$ is 2-connected.  
\end{lemma}

\begin{proof}
Otherwise, we may assume that $D$ has a block $B$ and a cut vertex $v\in V(B)$.
By the minimality of $(D,\sigma)$, we can extend $\phi_0$ to $(D-V(B-v),\sigma)$.
Let $f$ be a face of $B$ containing $v$ with minimum $d(f)$. If $d(f)\leq 11$, then Lemma \ref{bad cycle} implies that the boundary cycle of $f$ is good and therefore, we can extend the coloring of $v$ to a 3-coloring of $(f,\sigma)$ and further to a 3-coloring of $(B,\sigma)$ by the minimality of $(D,\sigma)$. If $d(f)\geq 12$, then insert into $f$ an arc $e$ with an arbitrary sign between the two neighbors of $v$ on $f$, creating a 3-face, say $T$. Note that $B+e\in \mathcal{G}$. Similarly, we can extend the coloring of $v$ to $(T,\sigma)$ and further to $(B+e,\sigma)$.
In either case, the resulting coloring of $(D,\sigma)$ is an extension of $\phi_0$, a contradiction. 
\end{proof}

\begin{lemma} \label{lem_min degree}
	Every internal vertex of $G$ has degree at least 3.
\end{lemma}

\begin{proof}
 If there exists an internal vertex $v$ of degree at most 2, then $\phi_0$ can extend to $(D-v,\sigma)$ and further to $(D, \sigma)$ by properly coloring $v$.    
\end{proof}

\begin{lemma}\label{lem_splitting path}
	Let $P$ be a splitting path of $U$, which divides $U$ into two cycles $U'$ and $U''$.
	If $2\leq |P| \leq 5$, then at least one of $U'$ and $U''$ has length from $|P|+1$ to $2|P|-1$.
	More precisely, since $G\in \mathcal{G}$,
	\begin{enumerate}[(1)]
		\setlength{\itemsep}{0pt} 
		\item if $|P|=2$, then at least one of $U'$ and $U''$ is a triangle.
		\item if $|P|=3$, then at least one of $U'$ and $U''$ is a 5-cycle.
		\item if $|P|=4$, then at least one of $U'$ and $U''$ is a 5- or 6-cycle.
		\item if $|P|=5$, then at least one of $U'$ and $U''$ is a 6- or 8-cycle.
	\end{enumerate}
\end{lemma}

\begin{proof}
	Suppose to the contrary that $|U'|, |U''| \geq 2|P|$.
	Since $U$ has length at most 13, we have $|U'|+|U''|=|U|+2|P|\leq 13+2|P|$.
	It follows that $2|P|\leq |U'|, |U''| \leq 13.$
	
	(1) Let $P=xyz$. Since $G$ has no 4-cycles, it follows that $5\leq |U'|, |U''| \leq 12.$
	By Lemma \ref{lem_min degree}, $y$ has a neighbor other than $x$ and $z$, say $y'$. The vertex $y'$ is internal since otherwise, $U$ is a bad cycle with a claw. W.l.o.g., let $y'$ lie inside $U'$. By Lemma \ref{lem_separating-cycle}, $U'$ is a bad 12-cycle. By Lemma \ref{bad cycle}, $U'$ has a claw, which together with $P$ forms a biclaw of $U'$, a contradiction.
	
	(2) Let $P=wxyz$. By Lemma \ref{lem_min degree}, we may let $x'$ and $y'$ be neighbors of $x$ and $y$ with $\{xx',yy'\}\cap E(P)=\emptyset$, respectively.
	If both $x'$ and $y'$ are external, then $U$ has a biclaw.
	Hence, we may assume that $x'\in \intt(U')$.
	So, $U'$ is a bad cycle.
	Since $G\in \mathcal{G}$,
	it follows that $U''$ is a 6-cycle and $U'$ has no 3-cell or 5-cell adjacent to $U''$.
	By Lemma \ref{bad cycle}, $U'$ must have a (6,6,6)-claw. But now there is no available location for $y'$.
	
	(3) Let $P=vwxyz$. 
	In this case, $8\leq |U'|, |U''| \leq 13.$
	Let $w'$, $x'$, and $y'$ be neighbors of $w$, $x$, and $y$ with $\{ww',xx',yy'\}\cap E(P)=\emptyset$, respectively.	
	W.l.o.g, let $x'$ lie in $\intt[U']$.
	If $x'\in V(U')$, then (2) implies that $xx'$ is a $(5,5)$-chord of $U'$.
    Further, the  (1) implies that $w',y'\notin V(U'')$ since otherwise, $G$ has a triangle adjacent to $U'$, which gives a 9-cycle.
    So, $w',y'\in \intt(U'')$. This results in a 4-cycle $[wxyy']$ if $w'=y'$; and a (5,5,5,8)-biclaw of $U''$ otherwise.
    For the latter case, $x$ is incident with three 5-faces, yielding a 9-cycle of $G$.
    It remains to assume that $x'\in \intt(U')$.
    In this case, $U'$ is a bad cycle, which implies that $|U''|=8$.
	Since $G\in \mathcal{G}$ and the statement (1), $wy\notin E(G)$ and neither $w'$ nor $y'$ is external.
	So, $w',y'\in \intt (U')$.  
	By Lemma \ref{bad cycle}, at least one of $w'$ and $y'$ coincides with $x'$, which yields a triangle adjacent to the 8-cycle $U''$, a contradiction.

	(4) Let $P=uvwxyz$. 
	In this case, $10\leq |U'|, |U''| \leq 13.$
	Let $v', w', x', y'$ be neighbors of $v, w, x, y$ with $\{vv', ww',xx',yy'\}\cap E(P)=\emptyset$, respectively.
	If $w'$ or $x'$ is external, w.l.o.g., say $w'\in V(U)\cap V(U')$, then by statements (2) and (3),
	$ww'$ is a $(3,5)$-chord of $U'$ if $w'=u$ and a $(5,5)$-chord of $U'$ otherwise, for which $|U'|\in \{6,8\}$, a contradiction.
	Hence, both $w'$ and $x'$ are internal. 
	Moreover, the assumption $|U'|+|U''|\leq 13+2|P|$ implies that not both $U'$ and $U''$ are bad cycles. W.l.o.g., let $w',x'\in \intt(U')$.
    Then $U'$ has either a claw or a biclaw. For the latter case, $[ww'x'x]$ is a 4-cycle of $G$, a contradiction. For the former case, $w'$ and $x'$ coincide, yielding a $(3,5,10)$- or $(3,5,11)$-claw of $U'$. Denote by $t$ the the remaining neighbor of $w'$ of this claw.
	Now both $uvww't$ and $tw'xyz$ are splitting 4-paths of $U$, but one of them together with $U$ forms no 5- or 6-cycles, contradicting  (3).
\end{proof}

\begin{lemma} \label{lem_string}
	For any integers $k$ and $t$ with $3\leq k\leq 12$ and $t\geq \lfloor \frac{k-1}{2}\rfloor$, and for any $k$-face $f$ of $G$ with $f\neq f_0$, the boundary of $f$ contains no $t$-strings.
\end{lemma}

\begin{proof}
	Suppose to the contrary that the boundary of $f$, say $C$, contains a $t$-string $L$.
	Let $D'=D-V(L)$ and let $U'$ be the boundary of the unbounded face of $D'$. 
	Note that the path $P=C-V(L)$ contains a splitting $q$-path $Q$ of $U$ with $q\geq 2$. 
	So, $q\leq |P|=k-t-1\leq \frac{k}{2}\leq 6$ by assumption.
	If $2\leq q\leq 5$, then by Lemma \ref{lem_splitting path}, $Q$ divides $U$ into two cycles, one contains $f$ and the other (say $C'$) has length from $q+1$ to $2q-1$. Clearly, $C'$ is a good cycle. By Lemma \ref{lem_min degree}, each middle vertex of $Q$ is incident with a chord of $C'$, which is impossible. 
	Hence, we may assume that $q=6$. In this case, $U'$ is a cycle.
    Note that $|U'|=|U|+k-2(t+1)\leq |U| \leq 13$. 
    If $U'$ is a bad cycle, then by Lemma \ref{lem_min degree}, each middle vertex of $Q$ is incident with a chord of $U'$, which is impossible.
    So, $U'$ is a good cycle. We can extend $\phi_0$ to $(D'[V(U')],\sigma)$ and further to $(D',\sigma)$ by the minimality of $(D,\sigma)$.  
\end{proof}

\begin{remark}\label{rem_switch}
	If $(D,\sigma')$ is an $S_3$-labelled graph switch-equivalent to $(D,\sigma)$, then $(G,\sigma')$ is a minimum counterexample to Theorem \ref{thm_main_extension} as well.
\end{remark}

\begin{lemma}\label{lem_n3face}
	$(D,\sigma)$ has no negative light 3-faces. 
\end{lemma}
\begin{proof}
	Otherwise, let $[uvw]$ be a negative light 3-face. 
	By the minimality of $(D,\sigma)$, $\phi_0$ can extend to $(D-\{u,v,w\},\sigma)$.
	Since $[uvw]$ is light, each vertex of $[uvw]$ has one neighbor already colored and so, it has two available colors for itself. Denote by $\sigma'$ the restriction of $\sigma$ on these available colors. 
	If there exists an edge $e$ of $[uvw]$, say $e=uv$, such that $\sigma'_e$ is not a full permutation,  then $u$ has an available color that is not in conflict  with any available  color for $v$. Assign $u$ with that color and consequently, we can properly color $w$ and $v$ in turn.
	Assume $\sigma'_c$ is a full permutation for each edge $c$ of $[uvw]$. 
    Since $[uvw]$ is negative in $(D,\sigma)$, it is also negative under $\sigma'$. Therefore, the vertices of $[uvw]$ can be properly colored.
\end{proof}

\begin{lemma}\label{pro_operation}
	Let $(D',\sigma)$ be a connected plane graph obtained from $(D,\sigma)$ by deleting a set of internal vertices and either identifying two other vertices without merging edges or inserting a new arc.
	If we
	\begin{enumerate}[($a$)]
		\item identify no two vertices of $U$ and create no edges connecting two vertices of $U$, and
		\item create no $9^-$-cycles,
	\end{enumerate}
	then $\phi_0$ can extend to $(D',\sigma)$.
\end{lemma}

\begin{proof}
	The item $(a)$ guarantees that $U$ is unchanged and bounds $D'$, and that $\phi_0$ is a 3-coloring of $(D'[V(U)],\sigma)$.
	By the item $(b)$, $G'$ is simple and $G'\in \cal{G}$.
	Hence, to extend $\phi_0$ to $(D',\sigma)$ by the minimality of $(D,\sigma)$, it suffices to show that $U$ is a good cycle in $G'$.
	
	Suppose to the contrary that $U$ is a bad cycle of $G'$, i.e., $U$ has a claw or biclaw, say $H$. 
	Assume the new vertex resulting from the identification is incident with $k$ cells of $H$. If $k= 0$, then $H$ is a claw or biclaw  in $G$. Since the operation does not merge edges, $k\neq 1.$ Therefore, $k= 2$.  
	It follows by Lemma \ref{bad cycle} that there is a $6^-$-cycle created, contradicting the item $(b)$. 
	For the case of inserting a new arc, say $e$, we can similarly deduce that both cells of $H$ incident with $e$ are created, yielding a similar contradiction as above.
\end{proof}

\begin{lemma} \label{lem_5face}
	$(G,\sigma)$ has no light 5-faces.
\end{lemma}
\begin{proof}
	Otherwise, let $f=[v_1v_2\ldots v_5]$ be a light 5-face. If $f$ is negative,
	then a similar proof as for Lemma \ref{lem_n3face} shows that $\phi_0$ can extend to $(D,\sigma)$. Assume that $f$ is positive.
	By Remark \ref{rem_switch}, we may choose $(D,\sigma)$ so that the edges incident with $v_i$ are all positive for each $1\leq i\leq 5$.
	Since $G\in \mathcal{G}$, $f$ has a vertex incident with two $8^+$-faces, w.l.o.g., say $v_1$.
	For $1\leq i\leq 5$, denote by $v_i'$ the remaining neighbor of $v_i$.
    Let $D'$ be obtained from $D$ by removing $V(f)$ and inserting a positive arc $v_2'v_5'$.
	
	We will show that both items of Lemma \ref{pro_operation} hold true for this graph operation.
	If a $9^-$-cycle is created, 
	then this cycle corresponds to a $8^-$-path between $v_2'$ and $v_5'$ in $G$, which together with $v_5'v _5v_1v_2v_2'$ forms a $12^-$-cycle, say $C$.
	If $v_3,v_4 \in \intt(C)$, then $C$ is a bad 12-cycle with a biclaw, a contradiction to Lemma \ref{bad cycle}.
	Thus either $v_1'\in V(C)$ or $v_1'\in \intt(C)$.
	In the former case, since $v_1$ is incident with two $8^+$-faces but $|C|\leq 12$, a contradiction can be derived. 
	In the latter case, $C$ is a bad 12-cycle containing two $8^+$-faces inside, a contradiction to Lemma \ref{bad cycle}.
	Therefore, $(b)$ holds.
	If both $v_2'$ and $v_5'$ are external,
	then $v_5'v _5v_1v_2v_2'$ is a splitting 4-path of $U$ in $G$, which divides $U$ into two cycles, one has length 5 or 6 by Lemma \ref{lem_splitting path}. It follows that inserting the arc $v_2'v_5'$ creates a cycle of length 6 or 7, contradicting $(b)$.  Hence,  $(a)$ holds.
		
	By Lemma \ref{pro_operation}, $\phi_0$ can extend to $(D',\sigma)$. Note that $v_5'$ and $v_2'$ received distinct colors.
	So, $v_1'$ is of color different from at least one of $v_5'$ and $v_2'$, w.l.o.g., say $v_5'$. Color $v_1$ the same as $v_5'$ and consequently, the resulting coloring can extend to $(D,\sigma)$ in the order $v_2\rightarrow v_3\rightarrow v_4\rightarrow v_5$.  
\end{proof}

\begin{lemma} \label{lem_good_path}
If $uvxyz$ is a path of $(G,\sigma)$ with $x$ and $y$ being two internal 3-vertices and $uv$ being contained in a bad 3-face $[uvw]$, then $z$ is neither an internal 3-vertex nor a neighbor of $x$ with $[xyz]$ being non-special. 
In particular, the outer neighbor of a bad vertex is not bad.
\end{lemma}

\begin{proof}
	Assume first that $z$ is an internal 3-vertex.
    Denote by $x'$ the remaining neighbor of $x$ and w.l.o.g., let $wvxx'$ be on the boundary of a face.
    Denote by $u'$ and $w'$ the remaining neighbors of $u$ and $w$, respectively.
    Since $[uvw]$ is positive, by Remark \ref{rem_switch}, we may choose $(D,\sigma)$ so that all the edges incident with $u$, $v$, or $x$ are positive.
    Let $D'$ be obtained from $D$ by removing $u,v,x,y,z$ and identifying $x'$ with $u'$. 
    
 	We will show that both items of Lemma \ref{pro_operation} hold true for this graph operation. 
 	If a $9^-$-cycle is created, 
 	then this cycle corresponds to a $9^-$-path between $x'$ and $u'$ in $G$, which together with $u'uvxx'$ forms a $13^-$-cycle, say $C$.
 	If $y,z \in \intt(C)$, then $C$ is a bad 13-cycle with a (5,5,5,8)-biclaw by Lemma \ref{bad cycle}, for which $[uvw]$ is adjacent to a 5- or 8-cell, contradicting the assumption that $[uvw]$ is a bad 3-face.
 	So by planarity, either $w'\in V(C)$ or $w'\in \intt(C)$.
 	For the former case, $w$ is incident with two $10^+$-faces but $|C|\leq 13$, a contradiction. 
 	For the latter case, $C$ is a bad cycle containing two $10^+$-faces inside, contradicting Lemma \ref{bad cycle}.
 	Therefore, the item $(b)$ holds.
 	If the item $(a)$ fails, then $u'uvxx'$ is contained in a splitting $5^-$-path of $U$ in $G$.
 	By Lemma \ref{lem_splitting path}, this splitting path divides $U$ into two cycles, one of which has length at most 8. This yields that identifying $u'$ and $x'$ creates a $4^-$-cycle, contradicting the item $(b)$. Hence, the item $(a)$ holds.
 	
 	By Lemma \ref{pro_operation}, $\phi_0$ can extend to $(D',\sigma)$ and further to $(D,\sigma)$ as follows. Since all the vertices of $uvxyz$ have degree 3, we can properly color $z$, $y$, and $x$ in order. Since all the edges incident with $u$, $v$, or $x$ are positive, $u'$ and $x$ must be of distinct colors and therefore, we can always properly color $u$ and $v$. 
 	
 	Next, let $z$ be a neighbor of $x$ with $[xyz]$ being non-special. The argument is almost same as the previous case. In this case, $x'$ coincides with $z$. So, $z$ will not be removed in the graph operation. The remaining difference is that, since $[xyz]$ is non-special, $y$ is incident with two $10^+$-faces, which also yields that $y\notin \intt(C)$ in the proof for the item $(b)$.
\end{proof}

\subsection{Discharging in $G$}\label{secch}
In what follows, let $V$, $E$, and $F$ be the set of vertices, edges, and faces of $G$, respectively. 
For each $x\in V\cup F$,
the \textit{initial charge} $ch(x)$ of $x$ is defined as 
\begin{equation*}\label{eq_def_initial_charge}
ch(x)=
\begin{cases}
d(x)+4, \text{~if~} x=f_0;\\
d(x)-4, \text{~otherwise}.
\end{cases}
\end{equation*}

Move charges among elements of $V\cup F$ according to the following rules (called \textit{discharging rules}):
\begin{enumerate}[$R1.$]
  \setlength{\itemsep}{0pt}
  \item The unbounded face $f_0$ sends to each incident vertex charge $\frac{17}{13}$.
  \label{rule-ext-face} 
  
  \item Every non-special 3-face receives from each incident vertex charge $\frac{1}{3}$.
\label{rule-3face}  

  \item Let $[uvw]$ be a non-special 3-face. If $u$ is an internal 3-vertex and $v$ is not, then $v$ sends to $u$ charge $\frac{2}{15}$.
\label{rule_non-bad-triangle}
  
  \item Every special $5$-face sends to each adjacent 3-face charge 1.
  \label{rule-35face} 
  
  \item Every non-special 5-face sends to each incident internal 3-vertex charge $\frac{1}{4}$, and to each incident 2-vertex charge $\frac{1}{2}$.
  \label{rule-5face}
  
  \item Every $6^+$-face $f$ $(f\neq f_0)$ sends to each incident vertex charge $\frac{d(f)-4}{d(f)}$.
  \label{rule-6+face}  
  
  \item Every non-bad vertex sends to each adjacent bad vertex charge $\frac{2}{15}$.
  \label{rule_bad-triangle}  
 
  \item Let $u$ be a vertex adjacent to a string $s$ and let $f$ ($f\neq f_0$) be the face containing $s$. Then $u$ sends to each vertex of $s$ charge $\frac{5}{52}$ when $f$ is a non-special 5-face, and charge $\frac{2}{d(f)}-\frac{2}{13}$ when $6\leq d(f) \leq 12$.
  \label{rule-string}

\end{enumerate}

Let $ch^*(x)$ denote the \textit{final charge} of an element $x$ of $V\cup F$ after the discharging procedure.
By Euler's formula $|V|-|E|+|F|=2$ and the Handshaking Theorem $2|E|=\sum\limits_{v\in V}d(v)=\sum\limits_{f\in F}d(f)$, we can deduce from the definition of $ch(x)$ that 
\begin{align*}
\sum\limits_{x\in V\cup F}ch(x)
&=\sum\limits_{v\in V}(d(v)-4)+\sum\limits_{f\in F}(d(f)-4)+8\\ 
&=\sum\limits_{v\in V}d(v)-4|V|+\sum\limits_{f\in F}d(f)-4|F|+8\\
&=4(|E|-|V|-|F|)+8\\
&=0.
\end{align*}
Since we just move charges from one element to another, the sum of charges over $V\cup F$ remains the same. So, 
$$\sum\limits_{x\in V\cup F}ch^*(x)=\sum\limits_{x\in V\cup F}ch(x)=0.$$ 
However, in what follows, we will show that $\sum\limits_{x\in V\cup F}ch^*(x)>0$. This contradiction completes the proof of Theorem \ref{thm_main_extension}.

Recall that $\mathcal{H}(G)$ (simply, $\mathcal{H}$) is the union of all the special subgraphs of $G$. 
For $v\in V(G)$, denote by $\mathcal{H}(v)$ the set of special subgraphs containing $v$ and let $h(v)=|\mathcal{H}(v)|$. 
For $H\in \mathcal{H}$, the initial charge of $H$ is defined as 
\begin{equation*} \label{eq_ch_H}
ch(H)=\sum_{v\in V(H)}\frac{ch(v)}{h(v)}.
\end{equation*}
Hence,
\begin{equation}\label{eq_H_sum}
\sum_{H\in \mathcal{H}}ch(H)
	=\sum_{H\in \mathcal{H}}\sum_{v\in V(H)}\frac{ch(v)}{h(v)}
    =\sum_{v\in V(\mathcal{H})}\sum_{H\in \mathcal{H}(v)}\frac{ch(v)}{h(v)}
    =\sum_{v\in V(\mathcal{H})}\frac{ch(v)}{h(v)}(\sum_{H\in \mathcal{H}(v)}1)
	=\sum_{v\in V(\mathcal{H})}ch(v).
\end{equation}
Similarly, we define the final charge of $H$ as 
\begin{equation*}
	ch^*(H)=\sum_{v\in V(H)}\frac{ch^*(v)}{h(v)},
\end{equation*} which similarly yields that 
$$\sum_{H\in \mathcal{H}}ch^*(H)=\sum_{v\in V(\mathcal{H})}ch^*(v).$$

For $x,y\in V\cup F$, denote by $ch(x\rightarrow y)$ the charge $x$ sends to $y$ by the discharging rules. 
\begin{claim}\label{claim_total_string}  
If $u$ is a vertex adjacent to a string $s$ of a face $f$ $(f\neq f_0)$ and $d(f)=k\geq 5$,  
then 
$$\sum_{x\in V(s)}ch(u\rightarrow x) \leq \frac{7}{26}$$ and 
\begin{equation} \label{eq_Ext_face_charge}
	ch(f\rightarrow u)-\sum_{x\in V(s)}ch(u\rightarrow x)\geq g(k)=
	\begin{cases}
		\frac{k-4}{13}, &\text{~when $k\in \{6,8,10,12\};$}\\
		\frac{k-4}{13}-(\frac{1}{k}-\frac{1}{13}), &\text{~when $k=11$};\\
		\frac{k-4}{k},  &\text{~when $k\geq 13$}.
	\end{cases}
\end{equation}
Clearly, $g(k)$ is a monotone increasing function on $k$.
\end{claim}

\begin{proof}
  Lemma \ref{lem_string} implies that $s$ has at most $\lfloor \frac{k-1}{2}\rfloor-1$ vertices. By $R\ref{rule-string}$,
  if $k=5$ then $\sum_{x\in V(s)}ch(u\rightarrow x)\leq \frac{5}{52} \times 1\leq \frac{7}{26};$
  if $k \geq 13$ then $\sum_{x\in V(s)}ch(u\rightarrow x)=0 \leq \frac{7}{26};$ and 
  if $6\leq k \leq 12$ then
  \begin{equation*}
  	\sum_{x\in V(s)}ch(u\rightarrow x)\leq (\frac{2}{k}-\frac{2}{13}) \times (\lfloor \frac{k-1}{2}\rfloor-1)\leq (\frac{2}{k}-\frac{2}{13}) \times \frac{k-3}{2}\leq \frac{7}{26},
  \end{equation*}
where the equality case of the last inequality holds for $k=6$.

Since $ch(f\rightarrow u)=\frac{k-4}{k}$ for $k\geq 6$ by $R\ref{rule-6+face}$ and $\sum_{x\in V(s)}ch(u\rightarrow x)\leq (\frac{2}{k}-\frac{2}{13}) \times (\lfloor \frac{k-1}{2}\rfloor-1)$ for $6\leq k\leq 12$, we can derive Formula \ref{eq_Ext_face_charge} by a direct computation.
\end{proof}

\begin{claim}\label{lem_charge}
	For each  $4^+$-vertex $u$ of $G$, we have 
	\begin{equation*}
		ch^*(u)
		\begin{cases}
			\geq \frac{1}{2}h(u), &\text{if $u$ is internal};\\
		    > \frac{11}{15}h(u), &\text{otherwise.}
		\end{cases}
	\end{equation*}
\end{claim}
\begin{proof}
Firstly, assume that $u$ is internal.
Denote by $r_1(u)$, $r_2(u)$, and $r_3(u)$ the number of non-special 3-faces, 8-faces, and 10$^+$-faces containing $u$, respectively. Denote by $b(u)$ the number of bad vertices adjacent to $u$. Since $G\in \mathcal{G}$, we can deduce that
\begin{equation}\label{eq_abr_relation}
	\begin{split}
		&r_1(u)+b(u)\leq r_3(u);\\ 
		&r_1(u)+b(u)+h(u)\leq r_2(u)+r_3(u). 		
	\end{split}
\end{equation} 
Notice that $u$ sends charge $\frac{1}{3}$ to each incident non-special 3-face $f$ by $R\ref{rule-3face}$,  charge $\frac{2}{15}$ to each internal 3-vertex on $f$ by $R\ref{rule_non-bad-triangle}$, and charge $\frac{2}{15}$ to each adjacent bad vertex by $R\ref{rule_bad-triangle}$.
Moreover, $u$ receives charge $\frac{1}{2}$ from each incident $8$-face and charge at least $\frac{3}{5}$ from each incident $10^+$-face by $R\ref{rule-6+face}$. It follows that
\begin{equation}\label{eq_charge}
	\begin{split}
		&ch^*(u)-\frac{1}{2}h(u)\\
		&\geq ch(u)-\frac{1}{2}h(u)-(\frac{1}{3}+\frac{2}{15}\times 2)r_1(u)-\frac{2}{15}b(u)+\frac{1}{2}r_2(u)+\frac{3}{5}r_3(u)\\
		&=ch(u)+\frac{1}{10}(r_3(u)-r_1(u)))+\frac{1}{2}(r_2(u)+r_3(u)-r_1(u)-h(u))-\frac{2}{15}b(u)\\
		&\geq \frac{1}{10}b(u)+\frac{1}{2}b(u)-\frac{2}{15}b(u)\\
		&=\frac{7}{15}b(u)\\
		&\geq 0,
	\end{split}
\end{equation}
where the second inequality uses Formula \ref{eq_abr_relation} and the assumption $ch(u)\geq 0$.

Next, assume that $u$ is external. We apply a similar argument as above. Let the counting for $r_1(u)$, $r_2(u)$, and $r_3(u)$ exclude $f_0$. Denote by $t(u)$ the number of non-special 5-face or $6^+$-face incident with $v$ and adjacent to $f_0$. Clearly, $0\leq t(u)\leq 2.$
Similarly as Formula \ref{eq_abr_relation}, we have
\begin{equation}\label{eq_abr_relation_ext}
		r_1(u)+h(u)+b(u)\leq r_2(u)+r_3(v)+1-t(u). 
\end{equation}
Moreover, since any two 3-faces or special 5-faces are edge-disjoint,
\begin{equation}\label{eq_ar_relation}
	2r_1(u)+2h(u)\leq d(u). 
\end{equation}
Note that $u$ receives charge $\frac{17}{13}$ from $f_0$ and sends total charge at most $\frac{7}{26}t(u)$ to 2-vertices by Claim \ref{claim_total_string}. Similarly as Formula \ref{eq_charge}, we have
\begin{equation}\label{eq_charge_ext}
	\begin{split}
		&ch^*(u)-\frac{11}{15}h(u)\\
		&\geq ch(u)-\frac{11}{15}h(u)-(\frac{1}{3}+\frac{2}{15}\times 2)r_1(u)-\frac{2}{15}b(u)+\frac{1}{2}(r_2(u)+r_3(v))-\frac{7}{26}t(u)+\frac{17}{13}\\
		&\geq ch(u)-\frac{11}{15}h(u)-\frac{3}{5}r_1(u)-\frac{2}{15}b(u)+\frac{1}{2}(r_1(u)+h(u)+b(u)+t(u)-1)-\frac{7}{26}t(u)+\frac{17}{13}\\
		&=d(u)-\frac{7}{30}h(u)-\frac{1}{10}r_1(u)+\frac{11}{30}b(u)+\frac{3}{13}t(u)-\frac{83}{26}\\
		&\geq d(u)-\frac{7}{30}(h(u)+r_1(u))-\frac{83}{26}\\
		&\geq d(u)-\frac{7}{30}\times \frac{d(u)}{2}-\frac{83}{26}\\
		&>0,
	\end{split}
\end{equation}
where the second, the forth, and the last inequalities use Formula \ref{eq_abr_relation_ext}, Formula \ref{eq_ar_relation}, and the fact $d(u)\geq 4$, respectively.
\end{proof}

\begin{claim} \label{claim_special-graph}
	For each $H\in \mathcal{H}$, we have 
	\begin{equation*}
		ch^*(H)
		\begin{cases}
			\geq 0, &\text{\text{if $H$ contains no external vertices};}\\
			> 0, &\text{otherwise}.
		\end{cases}
	\end{equation*}
\end{claim}
\begin{proof}
	Let $C=[v_1v_2\ldots v_6]$ be the 6-cycle of $H$ with $v_2v_6\in E(H)$. 
	We will show that for each $v\in V(H)$, 
	\begin{equation}\label{eq_H_int4_ext3}
		\frac{ch^*(v)}{h(v)}
		\begin{cases}
			> \frac{11}{15}, &\text{~if $v$ is an external $4^+$-vertex}; \\
			\geq \frac{1}{2}, &\text{~if $v$ is an internal $4^+$-vertex}; \\
	     	\geq \frac{8}{13}, &\text{~if $v$ is an external 3-vertex}; 		
		\end{cases}
	\end{equation}
	and if $v$ is an internal 3-vertex then 
	\begin{equation}\label{eq_bad_subgraph_3}
		\frac{ch^*(v)}{h(v)}\geq 
	\begin{cases}
		-\frac{2}{5}, &\text{~if $v\in \{v_2,v_6\}$}; \\		
		\frac{1}{15}, &\text{~if $v\in \{v_1,v_3,v_4,v_5\}$ and $v'$ is bad}; \\
		\frac{1}{5}, &\text{~if $v=v_1$ and $v'$ is not bad}; \\
		0, &\text{~if $v\in \{v_3,v_4,v_5\}$ and $v'$ is not bad}, \\
	\end{cases}
    \end{equation}
    where  $v'$ is the outer neighbor of $v$.
	
    For $d(v)\geq 4$, Formula \ref{eq_H_int4_ext3} follows directly from Claim \ref{lem_charge}. Assume that $d(v)=3$.
    In this case, $v$ is contained in precisely one special subgraph, i.e., $h(v)=1$.
    If $v$ is external, then $v$ is incident with $f_0$, $H$, and an $8^+$-face. Hence, $ch^*(v)/h(v)\geq ch(v)+\frac{17}{13}+\frac{4}{13}=\frac{8}{13}$ by the rule $R\ref{rule-ext-face}$ and Claim \ref{claim_total_string}, as desired.
    Next, assume that $v$ is internal. Note that only the rules $R\ref{rule-6+face}$ and $R\ref{rule_bad-triangle}$ might make $v$ move charge around. 
    For $v\in \{v_2,v_6\}$, $v$ is incident with one $10^+$-face, which sends to $v$ charge at least $\frac{3}{5}$ by $R\ref{rule-6+face}$. So, $ch^*(v)/h(v)\geq d(v)-4+\frac{3}{5}=-\frac{2}{5}$, as desired.
    Next, let $v\in\{v_1,v_3,v_4,v_5\}$. If $v'$ is a bad vertex, then $v$ sends to $v'$ charge $\frac{2}{15}$ by $R\ref{rule_bad-triangle}$, and $v$ is incident with two $10^+$-faces, which send to $v$ a total charge at least $\frac{3}{5}\times 2$ by $R\ref{rule-6+face}$. Hence, $ch^*(v)/h(v)\geq d(v)-4-\frac{2}{15}+\frac{3}{5}\times 2=\frac{1}{15}$, as desired.
    Next, assume that $v'$ is not a bad vertex. It is easy to verify that $ch^*(v)/h(v)\geq d(v)-4+\frac{3}{5}\times 2=\frac{1}{5}$ when $v=v_1$, and $ch^*(v)/h(v)\geq d(v)-4+\frac{1}{2}\times 2=0$ when $u\in \{v_3,v_4,v_5\}$. This proves Formula \ref{eq_bad_subgraph_3}.

We will verify the negativeness of the final charge of $H$ by using Formulas \ref{eq_H_int4_ext3} and  \ref{eq_bad_subgraph_3}.

Firstly, assume that $H$ contains a 2-vertex, say $v_j$. Clearly, $ch^*(v_j)=-2+\frac{17}{13}=-\frac{9}{13}$. Lemma \ref{lem_splitting path} implies that $j\in \{3,4,5\}$ and the path $v_{j-1}v_jv_{j+1}$ is the common part of $H$ and $U$. For $j=4$, if $d(v_2)=3$, then $v_3$ is incident with a $10^+$-face, yielding $ch^*(v_2)/h(v_2)+ch^*(v_3)/h(v_3)\geq -\frac{2}{5}+\min\{-1+\frac{17}{13}+\frac{6}{13},\frac{11}{15}\}=\frac{1}{3}$ by Claims \ref{claim_total_string} and \ref{lem_charge}; otherwise, $v_3$ is incident with an $8^+$-face, yielding $ch^*(v_2)/h(v_2)+ch^*(v_3)/h(v_3)\geq \frac{1}{2}+\min\{-1+\frac{17}{13}+\frac{4}{13},\frac{11}{15}\} \geq \frac{1}{3}$ by Claims \ref{claim_total_string} and \ref{lem_charge}. 
Therefore, $ch^*(H)=\sum_{v\in V(H)}ch^*(v)/h(v) \geq -\frac{9}{13}+\frac{1}{3} \times 2+\frac{1}{15}> 0$.
For $j\in \{3,5\}$, it is easy to verify that $ch^*(H)\geq \frac{1}{15}-\frac{2}{5}+0+\frac{8}{13}-\frac{9}{13}+\frac{11}{15}> 0.$

Next, assume that $H$ contains no 2-vertices.
If $H$ has at least two $4^+$-vertices or external 3-vertices, then $ch^*(H)\geq \min\{\frac{1}{2},\frac{11}{15},\frac{8}{13}\}\times 2-\frac{2}{5}\times 2+0\times 2>0,$ as desired; otherwise,
Lemma \ref{lem_5face} implies that $H$ contains precisely one $4^+$-vertex $w$ with $w\neq v_1.$ Note that the remaining vertices of $H$ are all internal 3-vertices and hence, they have no bad neighbors by Lemma \ref{lem_good_path}. 
If $w$ is external, then we have $ch^*(H)> \frac{11}{15}+\frac{1}{5}-\frac{2}{5}\times 2+0\times 2>0.$ Next, let $w$ be internal.
If $w\in \{v_2,v_6\}$, then $ch^*(H)\geq \frac{1}{2}-\frac{2}{5}+\frac{1}{5}+0\times 3\geq 0$; otherwise, each vertex of $\{v_3, v_5\} \setminus \{w\}$ is incident with a $10^+$-face and another $8^+$-face, yielding its final charge no less than $-1+\frac{3}{5}+\frac{1}{2}=\frac{1}{10}$. 
Hence,  $ch^*(H)\geq\frac{1}{2}-\frac{2}{5}\times 2+ \frac{1}{5}+\frac{1}{10}+0=0$.
\end{proof}

\begin{claim}\label{claim_int_vertex}
 For each $3^+$-vertex $v\in V(G)\setminus V(\mathcal{H})$, we have
 	\begin{equation*}
 	ch^*(v)
 	\begin{cases}
 		\geq 0, &\text{if $v$ is internal;}\\
 		> 0, &\text{otherwise}.
 	\end{cases}
 \end{equation*}
\end{claim}

\begin{proof}
We distinguish the following three cases.

Case 1: Let $v$ be an internal 3-vertex. Denote by $f_1$, $f_2$ and $f_3$ the faces containing $v$ with $d(f_1)\leq d(f_2)\leq d(f_3)$.

Firstly, assume that $d(f_1)=3$. It follows that $d(f_2)\geq 10$, since $G\in \mathcal{G}$ and $v\notin V(\mathcal{H})$. Hence, $v$ receives from each of $f_2$ and $f_3$ charge at least $\frac{3}{5}$ by $R\ref{rule-6+face}$ and sends to $f_1$ charge $\frac{1}{3}$ by $R\ref{rule-3face}$. Denote by $v'$ the outer neighbor of $v$. If $v$ is bad, then Lemma \ref{lem_good_path} implies that $v'$ is not bad and hence,  $v$ receives from $v'$ charge $\frac{2}{15}$ by $R\ref{rule_bad-triangle}$, giving  $ch^*(v)\geq d(v)-4+\frac{3}{5}\times 2-\frac{1}{3}+\frac{2}{15}=0$. Next, let $v$ be not bad. Since $v\notin V(\mathcal{H})$, $f_1$ is non-special. Hence, we can conclude from Lemma \ref{lem_n3face} that $f_1$ is not light, i.e., $f_1$ contains a vertex which is not an internal 3-vertex. 
By $R\ref{rule_non-bad-triangle}$, this vertex sends charge $\frac{2}{15}$ to $v$.
If $v'$ is not bad, then $v$ sends no charge to $v'$, giving $ch^*(v)\geq d(v)-4+\frac{3}{5}\times 2-\frac{1}{3}+\frac{2}{15}=0$ from above; otherwise, $v$ sends charge $\frac{2}{15}$ to $v'$ by $R\ref{rule_bad-triangle}$ and moreover, Lemma \ref{lem_good_path} implies that both the neighbors of $v$ on $f_1$ are not internal 3-vertices and together send to $v$ charge $\frac{2}{15}\times 2$ by $R\ref{rule_non-bad-triangle}$, giving $ch^*(v)\geq d(v)-4+\frac{3}{5}\times 2-\frac{1}{3}-\frac{2}{15}+\frac{2}{15}\times 2=0$. 

It remains to assume that $d(f_1)\geq 5$. Since $G\in \mathcal{G}$, we can deduce that either $d(f_1)=d(f_2)=d(f_3)=6$ or $d(f_3)\geq 8$. For the former case, we have $ch^*(v)= d(v)-4+\frac{1}{3}\times 3=0$ by $R\ref{rule-6+face}$.
For the latter case, if $v$ has no bad neighbors, then $ch^*(v)\geq d(v)-4+\frac{1}{4}\times 2+\frac{1}{2}=0$ by $R\ref{rule-5face}$ and $R\ref{rule-6+face}$. 
If $v$ has precisely one bad neighbor, then $d(f_2)\geq 10$, which gives $ch^*(v)\geq d(v)-4+\frac{1}{4}+\frac{3}{5}\times 2-\frac{2}{15}=\frac{19}{60}>0$ by  $R\ref{rule-5face}$, $R\ref{rule-6+face}$, and $R\ref{rule_bad-triangle}$.
If $v$ has more than one bad neighbor, then $d(f_1)\geq 10$, which gives $ch^*(v)\geq d(v)-4+\frac{3}{5}\times 3-\frac{2}{15}\times 3=\frac{2}{5}>0$ again by $R\ref{rule-5face}$, $R\ref{rule-6+face}$, and $R\ref{rule_bad-triangle}$.

Case 2: Let $v$ be an external 3-vertex. Denote by $f_1$ and $f_2$ the faces other than $f_0$ containing $v$ with $d(f_1)\leq d(f_2)$, and $v'$ the unique internal neighbor of $v$.
If $v'$ is bad, then $d(f_1)\geq 10$ and so, 
$ch^*(v)\geq ch(v)+\frac{17}{13}-\frac{2}{15}+\frac{6}{13}\times 2>0$ by $R\ref{rule-ext-face}, R\ref{rule_non-bad-triangle}$, and Claim \ref{claim_total_string}.
Next, let $v'$ be not bad.
If $d(f_1)=3$, then $d(f_2)\geq 10$, yielding $ch^*(v)\geq ch(v)+\frac{17}{13}-\frac{1}{3}-\frac{2}{15}+\frac{6}{13}>0$ by $R\ref{rule-ext-face}$, $R\ref{rule_bad-triangle}$, and Claim \ref{claim_total_string}.
If $d(f_1)\geq 5$, then we have $ch^*(v)\geq ch(v)+\frac{17}{13}+\min\{-\frac{5}{52},\frac{2}{13}\}\times 2> 0$ by $R\ref{rule-ext-face}, R\ref{rule-string}$, and Claim \ref{claim_total_string}.  

Case 3: Let $d(v)\geq 4$. Since $v\notin V(\mathcal{H})$, $h(v)=0$.
So, the conclusion of this claim follows directly from Claim \ref{lem_charge}. 
\end{proof}

\begin{claim}
	$ch^*(v)\geq 0$ for each 2-vertex $v\in V(G) \setminus V(\mathcal{H})$. 
\end{claim}
\begin{proof}
	By Lemma \ref{lem_min degree}, $v\in V(f_0)$.
	Let $f$ be the face containing $v$ other than $f_0$. 
	Clearly, $v$ receives charge $\frac{17}{13}$ from $f_0$ by $R$\ref{rule-ext-face}.
	If $d(f)\geq 6$, then $v$ receives charge $\frac{d(f)-4}{d(f)}$ from $f$ by $R$\ref{rule-6+face}, which gives 
	\begin{equation}\label{eq_2vertex_1}
		ch^*(v)\geq d(v)-4+\frac{17}{13}+\frac{d(f)-4}{d(f)}=4(\frac{1}{13}-\frac{1}{d(f)}).
	\end{equation}
	Hence, $ch^*(v)\geq 0$ when $d(f)\geq 13$.
	Denote by $s$ the string containing $v$. For $6\leq d(f)\leq 12$, 
    the two vertices adjacent to $s$ send to $v$ total charge $(\frac{2}{d(f)}-\frac{2}{13})\times 2$ by $R$\ref{rule-string}, which strengthens Formula \ref{eq_2vertex_1} as 
	$ch^*(v)\geq 4(\frac{1}{13}-\frac{1}{d(f)})+(\frac{2}{d(f)}-\frac{2}{13})\times 2=0.$
	Lemma \ref{lem_string} implies that $d(f)\geq 5$. Hence, it remains to assume that $d(f)=5$. Since the assumption $v\notin V(\mathcal{H})$, $f$ is non-special. So, $ch^*(v)=ch(v)+\frac{17}{13}+\frac{1}{2}+\frac{5}{52}\times 2=0$ by the rules $R\ref{rule-ext-face}$, $R\ref{rule-5face}$, and $R\ref{rule-string}$.
\end{proof}

\begin{claim}\label{claim_f0}
	$ch^*(f)\geq 0$ for each face $f$ of $G$.
\end{claim}
\begin{proof}
	Note that only the rule $R$\ref{rule-ext-face} makes $f_0$ move charge out. If $f=f_0$, then $ch^*(f)=ch(f_0)-\frac{17}{13}d(f_0)=d(f_0)+4-\frac{17}{13}d(f_0)=4-\frac{4}{13}d(f_0)\geq 0$, since $d(f_0)\leq 13$. Let us next assume that $f\neq f_0$.
	Since $G\in \mathcal{G}$, $d(f)\notin \{4,7,9\}$. We may distinguish the following three cases.
	
	(\rmnum{1})  Let $d(f)=3$. Since $G\in \mathcal{G}$, $f$ is adjacent to at most one 5-face. If $f$ is adjacent to a 5-face, i.e., $f$ is special, then $f$ receives from this 5-face charge $1$ by  $R$\ref{rule-35face}, giving $ch^*(f)\geq d(f)-4+1=0$; otherwise, $f$ receives from each incident vertex charge $\frac{1}{3}$ by $R$\ref{rule-3face}, giving $ch^*(f)= d(f)-4+\frac{1}{3}\times 3=0$. 
	
	(\rmnum{2}) Let $d(f)=5$.  
	Since $G\in \mathcal{G}$, $f$ is adjacent to at most one 3-face. 
	If $f$ is adjacent to a 3-face, i.e., $f$ is special, then $f$ sends charge $1$ to this 3-face by $R$\ref{rule-35face}, giving $ch^*(f)\geq d(f)-4-1=0$. Next, assume that $f$ is not special. Lemma \ref{lem_5face} implies that $f$ is not light, i.e., $f$ contains at most four internal 3-vertices. If $f$ contains no 2-vertices, then $ch^*(f)\geq d(f)-4-\frac{1}{4}\times 4=0$ by $R$\ref{rule-5face}; otherwise, Lemma \ref{lem_string} implies that $f$ contains precisely one 2-vertex and consequently at most two internal 3-vertices, yielding that $ch^*(f)\geq d(f)-4-\frac{1}{2}- \frac{1}{4}\times 2=0$ by $R$\ref{rule-5face}.

	(\rmnum{3}) Let $d(f)\in \{6,8\}$ or $d(f)\geq 10$. 
	Since $R$\ref{rule-6+face} is the only rule making $f$ move charge out, we have $ch^*(f)=d(f)-4-\frac{d(f)-4}{d(f)}\times d(f)=0.$
\end{proof}

As a counterexample to Theorem \ref{thm_main_extension}, $(G,\sigma)$ must contain an external $3^+$-vertex. So by Claims \ref{claim_special-graph}--\ref{claim_f0}, we have $\sum_{x\in V\cup F}ch^*(x)>0$, completing the proof of Theorem \ref{thm_main_extension}.

\section{Acknowledgement}
Yingli Kang is supported by NSFC (Grant No. 11901258) and ZJNSF (Grant No. LY22A010016).
Ligang Jin is supported by ZJNSF (Grant No. LY20A010014) and NSFC (Grant No. 11801522 and U20A2068).

\end{document}